%% file: main_plain.tex
\documentclass[11pt]{article}

\usepackage[T1]{fontenc}
\usepackage[utf8]{inputenc}
\usepackage{lmodern}
\usepackage[margin=1in]{geometry}
\usepackage{microtype}
\usepackage{amsmath,amssymb,amsthm,mathtools,bm}
\usepackage{booktabs,array,tabularx}
\usepackage{enumitem}
\usepackage[numbers,sort&compress]{natbib}
\usepackage[hidelinks]{hyperref}
\usepackage[nameinlink,noabbrev]{cleveref}

\setlist[enumerate]{leftmargin=.5in}
\setlist[itemize]{leftmargin=.5in}

\input{macros}

\allowdisplaybreaks
\numberwithin{equation}{section}

\theoremstyle{plain}
\newtheorem{theorem}{Theorem}[section]
\newtheorem{proposition}[theorem]{Proposition}
\newtheorem{lemma}[theorem]{Lemma}
\newtheorem{corollary}[theorem]{Corollary}

\theoremstyle{definition}
\newtheorem{assumption}[theorem]{Assumption}

\theoremstyle{remark}
\newtheorem{remark}[theorem]{Remark}

\crefname{theorem}{Theorem}{Theorems}
\Crefname{theorem}{Theorem}{Theorems}
\crefname{proposition}{Proposition}{Propositions}
\Crefname{proposition}{Proposition}{Propositions}
\crefname{lemma}{Lemma}{Lemmas}
\Crefname{lemma}{Lemma}{Lemmas}
\crefname{corollary}{Corollary}{Corollaries}
\Crefname{corollary}{Corollary}{Corollaries}
\crefname{remark}{Remark}{Remarks}
\Crefname{remark}{Remark}{Remarks}
\crefname{assumption}{Assumption}{Assumptions}
\Crefname{assumption}{Assumption}{Assumptions}
\crefname{equation}{equation}{equations}
\Crefname{equation}{Equation}{Equations}
\crefname{section}{Section}{Sections}
\Crefname{section}{Section}{Sections}
\crefname{subsection}{Section}{Sections}
\Crefname{subsection}{Section}{Sections}
\crefname{table}{Table}{Tables}
\Crefname{table}{Table}{Tables}

\newenvironment{keywords}{%
  \par\smallskip\noindent\textbf{Keywords. }%
}{\par\smallskip}

\newenvironment{MSCcodes}{%
  \par\smallskip\noindent\textbf{MSC2020. }%
}{\par\smallskip}

\title{Reinforcement Learning, Optimal Control, and Bayesian Filtering in Data Assimilation}
\author{Abed Hammoud\\
Civil and Environmental Engineering, Princeton University\\
\texttt{ah1389@princeton.edu}}
\date{\today}

\hypersetup{
  pdftitle={Reinforcement Learning, Optimal Control, and Bayesian Filtering in Data Assimilation},
  pdfauthor={Abed Hammoud}
}

\begin{document}

\maketitle

\begin{abstract}
We give a finite-horizon variational formulation that places Bayesian filtering and smoothing, variational data assimilation, KL-regularized control, and Kalman-type methods inside one mathematically explicit hierarchy. For a discrete-time hidden Markov model and any admissible one-step candidate law $q_t$, we prove $\mathcal J_t(q_t)=\mathbb E_{q_t}\!\left[-\log p(y_t\mid X_t)\right]+\KL(q_t\|p_t^f)=\KL(q_t\|p_t^a)-\log p(y_t\mid y_{0:t-1})$, and, for any admissible path law $q$, $\mathcal J_{\mathrm{path}}(q)=\mathbb E_q\!\left[-\sum_{t=0}^{T}\log p(y_t\mid X_t)\right]+\KL(q\|p(\xtraj))=\KL(q\|p(\xtraj\mid \ytraj))-\log p(\ytraj)$. These identities determine the evidence as the global infimum and make the analysis and smoothing posteriors the unique minimizers whenever those posterior laws belong to the admissible classes. This separates targets that are often conflated: strong- and weak-constraint 4D-Var are MAP estimators under the stated Gaussian assumptions; KL-regularized control recovers the Bayesian posterior only when the passive dynamics, likelihood cost, temperature, and a restrictive representability condition on the policy class are all matched correctly; and the linear-Gaussian specialization yields the Kalman analysis exactly. The ensemble Kalman filter then appears as a Gaussian and finite-ensemble approximation to the forecast-to-analysis map, exact only in the linear-Gaussian infinite-ensemble limit. This framework also clarifies RMSE-based RL data assimilation: such rewards may define effective estimators or pseudo-posteriors, but not exact posterior recovery unless they realize the likelihood-plus-KL objective. 
\end{abstract}



\input{01_introduction}
\input{02_state_space}
\input{03_var_da}
\input{04_posterior_variational}
\input{05_rl_control}
\input{06_linear_gaussian}
\input{07_enkf}
\input{08_hammoud}
\input{10_conclusion}

\appendix
\input{appendix}

\bibliographystyle{plainnat}
\bibliography{refs}

\end{document}

%% file: macros.tex
\newcommand{\R}{\mathbb{R}}
\newcommand{\E}{\mathbb{E}}

\newcommand{\KL}{\operatorname{KL}}
\newcommand{\Cov}{\operatorname{Cov}}

\newcommand{\tr}{\operatorname{tr}}
\newcommand{\Id}{I}

\newcommand{\dd}{\,\mathrm d}

\newcommand{\argmin}{\operatorname*{arg\,min}}

\newcommand{\xtraj}{x_{0:T}}
\newcommand{\ytraj}{y_{0:T}}
\newcommand{\atraj}{a_{0:T-1}}
\newcommand{\Xtraj}{X_{0:T}}

\newcommand{\qpi}{q_\pi}
\newcommand{\norm}[1]{\left\lVert #1 \right\rVert}
\newcommand{\abs}[1]{\left\lvert #1 \right\rvert}
\newcommand{\inner}[2]{\left\langle #1, #2 \right\rangle}
\newcommand{\set}[1]{\left\{ #1 \right\}}
\newcommand{\bracks}[1]{\left[ #1 \right]}
\newcommand{\parens}[1]{\left( #1 \right)}

\newcommand{\ind}{\mathbf{1}}
\newcommand{\transpose}{^{\mathsf T}}

\newcommand{\mc}{\mathcal}

%% file: 01_introduction.tex
\section{Introduction}
\label{sec:introduction}

Data assimilation combines a dynamical model with partial and noisy observations in order to estimate the state of a system. In geophysical applications the problem is sequential, high dimensional, and often nonlinear. Three mathematical languages dominate the subject. The first is \emph{Bayesian filtering and smoothing}, where the target is a posterior law conditioned on data. The second is \emph{variational data assimilation}, where one minimizes a deterministic mismatch functional over states or trajectories. The third is the \emph{Kalman and ensemble Kalman} language, where one propagates means, covariances, and their ensemble approximations through forecast and analysis updates. These languages are closely related, but they are not interchangeable, and reinforcement-learning terminology can make the distinctions even harder to state precisely.

Historically, variational data assimilation emerged from optimal-control ideas and adjoint equations in meteorology \citep{ledimet1986,talagrandcourtier1987}, while Bayesian interpretations of numerical weather prediction analysis were made explicit by Lorenc \cite{lorenc1986}. Ensemble Kalman methods turned Gaussian filtering formulas into practical ensemble algorithms for large systems \citep{burgers1998,whitakerhamill2002,tippett2003,evensen2003}. Particle-filter viewpoints retained the goal of approximating the full nonlinear posterior while also clarifying the high-dimensional degeneracy barrier in geoscience data assimilation \citep{fearnheadkuensch2018,vanleeuwen2019}. On a different track, linearly solvable Markov decision processes, KL-control, and control-as-inference showed that entropy-regularized control objectives can be rewritten as inference or KL minimization problems \citep{toussaintstorkey2006,todorov2006,todorov2008,todorov2009,kappen2012,rawlik2012,levine2018}. Controlled interacting particle systems, feedback particle filters, and Schr\"odinger or optimal-transport viewpoints supplied further bridges between filtering, control, and data assimilation \citep{yang2013,taghvaeimehta2023,reich2019}. Recent reinforcement-learning-based data assimilation papers, including Hammoud et al.\ \cite{hammoud2024}, make these bridges practically salient, but also increase the temptation to transfer posterior-exact language to objectives that do not actually compute posteriors.

Our goal is conceptual rather than algorithmic. We do not propose a new approximate filter, prove finite-sample optimality of EnKF-type methods, or claim posterior recovery for generic RL rewards. Instead, we give a mathematically explicit synthesis of the exact finite-horizon statements that survive across these literatures and state explicitly where the usual slogans stop being theorems.

The contributions are the following.
\begin{enumerate}[label=\arabic*.]
\item \textbf{Posterior-as-variational theorems.} We prove exact one-step and path-space identities for hidden Markov models, \cref{thm:one-step-variational,thm:path-variational}, and then state the corresponding unique-minimizer conclusions under explicit posterior-admissibility assumptions.
\item \textbf{4D-Var as MAP, not posterior-law optimization.} Under the stated Gaussian assumptions, strong- and weak-constraint 4D-Var are negative-log-posterior minimization problems and therefore compute MAP estimators rather than posterior means or posterior laws; see \cref{thm:strong-4dvar,prop:weak-4dvar}.
\item \textbf{A mathematically honest control-as-inference statement.} KL-regularized control recovers the Bayesian posterior only when the passive dynamics, likelihood cost, temperature, and a restrictive representability condition on the admissible policy class are all matched correctly; see \cref{thm:gibbs-control,cor:policy-posterior,prop:soft-bellman}. Outside that calibrated setting one obtains either a different Gibbs law or a restricted-family projection.
\item \textbf{Kalman exactness and EnKF approximation.} In the linear-Gaussian case the unrestricted optimizer is the Kalman analysis posterior, and Gaussian minimization reproduces the classical mean and covariance formulas; see \cref{thm:kalman-posterior,prop:gaussian-functional}. EnKF-type methods then appear as Gaussian and finite-ensemble approximations to the forecast-to-analysis map, not as generic exact posterior solvers at finite ensemble size; see \cref{sec:enkf}.
\item \textbf{Clarification for RMSE-based RL data assimilation.} RMSE-style rewards can define useful estimation or control objectives, but they do not imply exact posterior recovery unless they realize the same likelihood-plus-KL structure as the Bayesian objective; see \cref{thm:rl-one-step,sec:hammoud}.
\end{enumerate}

The backbone of the paper is the pair of identities proved in \cref{thm:one-step-variational,thm:path-variational}. The one-step filtering identity reads
\[
\mathcal J_t(q_t)
=
\mathbb E_{q_t}\!\left[-\log p(y_t\mid X_t)\right]
+
\KL(q_t\|p_t^f)
=
\KL(q_t\|p_t^a)-\log p(y_t\mid y_{0:t-1}),
\]
for every admissible candidate analysis law $q_t$, while the path-space analogue is
\[
\mathcal J_{\mathrm{path}}(q)
=
\mathbb E_q\!\left[-\sum_{t=0}^{T}\log p(y_t\mid X_t)\right]
+
\KL(q\|p(\xtraj))
=
\KL(q\|p(\xtraj\mid \ytraj))-\log p(\ytraj).
\]
These identities identify the evidence as the variational infimum and make the posterior the unique minimizer whenever the posterior itself belongs to the admissible family. Restricting that family then recovers familiar approximations: MAP in the point-mass or zero-variance limit, Gaussian variational approximations in Gaussian families, and ensemble approximations when the posterior update is replaced by sample-based Gaussian surrogates. Core notation is summarized in \cref{tab:notation}.

Four qualifications are essential for the rest of the paper. First, the minimizer statements in \cref{sec:variational-posterior,sec:rl-control} require posterior or Gibbs admissibility in addition to mere normalization of the evidence. Second, exact posterior recovery by a policy class is a reachability statement on state-action laws and is therefore not automatic under fixed transition kernels. Third, EnKF exactness means exactness only in the linear-Gaussian infinite-ensemble regime; finite-ensemble stochastic and deterministic square-root filters remain Gaussian or ensemble approximations. Fourth, except for the deterministic strong-constraint formulation in \cref{sec:4dvar}, the main theorems are written in density form for finite-dimensional Euclidean state spaces.

The results below are intentionally finite-horizon and discrete-time. This keeps the main theorems exact and transparent, while still covering the standard filtering, smoothing, 4D-Var, Kalman, and EnKF settings used in data assimilation. \Cref{sec:state-space,sec:4dvar,sec:variational-posterior} establish the notation and the posterior-as-variational identities; \cref{sec:rl-control,sec:linear-gaussian,sec:enkf} develop the control, Kalman, and EnKF specializations; 
\cref{sec:hammoud,sec:conclusion} then interpret RMSE-based RL data assimilation, and close the argument. Continuous-time remarks are deferred to the appendix.

%% file: 02_state_space.tex
\section{State-space models, filtering, smoothing, and notation}
\label{sec:state-space}

We work on a fixed finite horizon $\set{0,\dots,T}$. Uppercase letters denote random variables, and lowercase letters denote realizations. For simplicity, all hidden states and observations take values in Euclidean spaces. Except where explicitly noted otherwise, notably the deterministic strong-constraint model in \cref{sec:4dvar}, the probability laws and conditional laws used in the filtering, smoothing, variational-law, Kalman, EnKF, and RL sections admit Borel densities with respect to Lebesgue measure. This covers the standard finite-dimensional hidden-Markov-model setting used in the density-based parts of the paper. All logarithms are natural logarithms.

\begin{assumption}[Discrete-time hidden Markov model]
\label{ass:ssm}
Let $X_t\in\R^n$ and $Y_t\in\R^m$. The hidden process and observation process satisfy
\begin{align}
X_0 &\sim p(x_0), \label{eq:ssm-initial}\\
X_{t+1}\mid X_t=x_t &\sim p(x_{t+1}\mid x_t), \qquad t=0,\dots,T-1, \label{eq:ssm-transition}\\
Y_t\mid X_t=x_t &\sim p(y_t\mid x_t), \qquad t=0,\dots,T. \label{eq:ssm-observation}
\end{align}
Conditionally on $\Xtraj$, the observations are independent, and $X$ is Markov. For the realized observation sequence $\ytraj$, the one-step evidence
\begin{equation}
Z_t := p(y_t\mid y_{0:t-1}) = \int p(y_t\mid x_t)p_t^f(x_t)\dd x_t,
\end{equation}
and the path-space evidence
\begin{equation}
Z_{0:T} := p(\ytraj) = \int p(\xtraj)\prod_{t=0}^{T} p(y_t\mid x_t)\dd \xtraj,
\end{equation}
exist and satisfy $0<Z_t<\infty$ and $0<Z_{0:T}<\infty$.
\end{assumption}

The prior path density induced by \cref{eq:ssm-initial,eq:ssm-transition} is
\begin{equation}
p(\xtraj) = p(x_0)\prod_{t=0}^{T-1} p(x_{t+1}\mid x_t).
\label{eq:path-prior}
\end{equation}
The corresponding smoothing posterior is
\begin{equation}
p(\xtraj\mid \ytraj)
=
\frac{1}{p(\ytraj)}
p(\xtraj)\prod_{t=0}^{T} p(y_t\mid x_t).
\label{eq:path-posterior}
\end{equation}

For filtering, we use the standard forecast and analysis densities
\begin{align}
p_t^f(x_t) &:= p(x_t\mid y_{0:t-1}), \label{eq:forecast-density}\\
p_t^a(x_t) &:= p(x_t\mid y_{0:t}). \label{eq:analysis-density}
\end{align}
The superscripts ``$f$'' and ``$a$'' refer to \emph{forecast} and \emph{analysis}. The distribution $p_t^a$ is the one-step filtering posterior after assimilating observation $y_t$.

\begin{table}[t]
\centering
\caption{Core notation used throughout the manuscript.}
\label{tab:notation}
\begin{tabularx}{\textwidth}{@{}>{\raggedright\arraybackslash}p{0.25\textwidth}X@{}}
\toprule
Symbol or term & Meaning \\
\midrule
$\xtraj$, $\ytraj$ & State trajectory and realized observation sequence on the time window $\set{0,\dots,T}$. \\
$p_t^f$, $p_t^a$ & Forecast density $p(x_t\mid y_{0:t-1})$ and analysis posterior $p(x_t\mid y_{0:t})$. \\
$q_t$ & Admissible candidate analysis law at time $t$. \\
$p(\xtraj)$, $p(\xtraj\mid \ytraj)$ & Prior path density and smoothing posterior on trajectories. \\
$\qpi$ & Policy-induced state-action path law under the passive dynamics; see \cref{eq:policy-path-law}. \\
forecast / analysis & Prediction before assimilating $y_t$ / posterior update after assimilating $y_t$. \\
passive dynamics & Reference uncontrolled dynamics, or reference state-action dynamics, relative to which KL regularization is measured. \\
MAP & Maximum a posteriori estimator, that is, any maximizer of a posterior density. \\
\bottomrule
\end{tabularx}
\end{table}

\begin{proposition}[Forecast-analysis recursion]
\label{prop:forecast-analysis}
Under \cref{ass:ssm}, the forecast and analysis densities obey
\begin{align}
p_0^f(x_0)
&=
p(x_0), \label{eq:forecast-initial}\\
p_t^f(x_t)
&=
\int p(x_t\mid x_{t-1})p_{t-1}^a(x_{t-1})\dd x_{t-1},
\qquad t=1,\dots,T,
\label{eq:forecast-recursion}\\
p_t^a(x_t)
&=
\frac{p(y_t\mid x_t)p_t^f(x_t)}{\int p(y_t\mid z_t)p_t^f(z_t)\dd z_t}
=
\frac{p(y_t\mid x_t)p_t^f(x_t)}{p(y_t\mid y_{0:t-1})}.
\label{eq:analysis-recursion}
\end{align}
\end{proposition}

\begin{proof}
\Cref{eq:forecast-initial} is the empty-history forecast. For $t\ge 1$, \cref{eq:forecast-recursion} is the Chapman--Kolmogorov prediction step. For \cref{eq:analysis-recursion}, Bayes' rule gives
\[
p(x_t\mid y_{0:t})
=
\frac{p(y_t\mid x_t,y_{0:t-1})p(x_t\mid y_{0:t-1})}{p(y_t\mid y_{0:t-1})}.
\]
By conditional independence, $p(y_t\mid x_t,y_{0:t-1})=p(y_t\mid x_t)$, which yields the stated formula.
\end{proof}

Two distinctions will be used repeatedly.
\begin{enumerate}[label=\arabic*.]
\item The \emph{analysis posterior} $p_t^a$ is a single-time filtering law. The \emph{smoothing posterior} $p(\xtraj\mid \ytraj)$ is a law on full trajectories.
\item A \emph{posterior mean} and a \emph{MAP estimator} are point summaries of a posterior law. They need not coincide, and neither summary determines the full posterior distribution in general.
\item A \emph{Gaussian variational approximation} is the minimizer of a restricted optimization over Gaussian laws, whereas a \emph{finite-ensemble approximation} is a random empirical measure built from finitely many particles. Neither object should be called the posterior unless a separate theorem identifies it with the posterior law.
\end{enumerate}

We now turn to variational data assimilation, which fits into this notation as an optimization of a negative log-posterior density under additional structural assumptions.

%% file: 03_var_da.tex
\section{Variational data assimilation and 4D-Var as MAP}
\label{sec:4dvar}

Variational data assimilation is often presented as a deterministic optimization method. Under the Gaussian background, model-error, and observation-error assumptions stated below, however, its objective is exactly a negative log-posterior \citep{lorenc1986,zupanski1997}. That is the precise sense in which 4D-Var is Bayesian. The resulting minimizer is a MAP estimator, not in general a posterior mean and not the posterior law itself. Outside these assumptions, the same deterministic objective may remain useful, but the MAP interpretation no longer follows automatically.

This section also marks the main exception to the density-based setup from \cref{sec:state-space}. In strong-constraint 4D-Var, the deterministic dynamics place the trajectory law on the model manifold generated by the initial condition, so the optimization is over $x_0$ or over trajectories constrained to that manifold rather than over Lebesgue densities on the full path space.

\subsection{Strong-constraint 4D-Var}

Assume that the model is perfect inside the assimilation window:
\begin{equation}
x_{t+1} = \mathcal M_t(x_t), \qquad t=0,\dots,T-1,
\label{eq:strong-model}
\end{equation}
so the full trajectory is determined by the initial state $x_0$. Let $\mathcal M_{0:t}$ denote the model composition mapping $x_0$ to $x_t$. Assume further that
\begin{align}
X_0 &\sim \mathcal N(x_b,B), \label{eq:background-law}\\
Y_t &= \mathcal H_t(X_t) + \eta_t, \qquad \eta_t\sim \mathcal N(0,R_t), \qquad t=0,\dots,T,
\label{eq:obs-operator}
\end{align}
with $B$ and each $R_t$ symmetric positive definite, and with independent observation errors.

\begin{theorem}[Strong-constraint 4D-Var objective as negative log-posterior]
\label{thm:strong-4dvar}
Under \cref{eq:strong-model,eq:background-law,eq:obs-operator}, the posterior density of $x_0$ given $\ytraj$ satisfies
\begin{equation}
p(x_0\mid \ytraj)
\propto
\exp\parens{-J_{\mathrm{sc}}(x_0)},
\end{equation}
where
\begin{equation}
J_{\mathrm{sc}}(x_0)
:=
\frac12 (x_0-x_b)\transpose B^{-1}(x_0-x_b)
\;+\;
\frac12\sum_{t=0}^{T}
\parens{y_t-\mathcal H_t(\mathcal M_{0:t}(x_0))}\transpose
R_t^{-1}
\parens{y_t-\mathcal H_t(\mathcal M_{0:t}(x_0))},
\label{eq:strong-4dvar-cost}
\end{equation}
Consequently, any minimizer of $J_{\mathrm{sc}}$ is a MAP estimator of the initial condition, and the corresponding trajectory $\set{\mathcal M_{0:t}(x_0)}_{t=0}^{T}$ is the associated strong-constraint MAP trajectory on the deterministic model manifold.
\end{theorem}

\begin{proof}
Because the model is deterministic, the likelihood of the observed sequence conditioned on $x_0$ factorizes as
\[
p(\ytraj\mid x_0)
=
\prod_{t=0}^{T} p\parens{y_t\mid \mathcal M_{0:t}(x_0)}.
\]
Under Gaussian observation errors,
\[
p\parens{y_t\mid \mathcal M_{0:t}(x_0)}
\propto
\exp\!\left(
-\frac12
\parens{y_t-\mathcal H_t(\mathcal M_{0:t}(x_0))}\transpose
R_t^{-1}
\parens{y_t-\mathcal H_t(\mathcal M_{0:t}(x_0))}
\right).
\]
The Gaussian background density satisfies
\[
p(x_0)
\propto
\exp\!\left(
-\frac12 (x_0-x_b)\transpose B^{-1}(x_0-x_b)
\right).
\]
Multiplying prior and likelihood and applying Bayes' rule yields the claimed posterior density with respect to Lebesgue measure on $\R^n$. Since the posterior density is proportional to $\exp(-J_{\mathrm{sc}}(x_0))$, maximizing the posterior over the support of that density is equivalent to minimizing $J_{\mathrm{sc}}$. The associated trajectory statement is the induced MAP statement on the deterministic model manifold $\set{(\mathcal M_{0:t}(x_0))_{t=0}^{T}:x_0\in\R^n}$.
\end{proof}

\begin{remark}[When the minimizer is a MAP estimator]
\label{rem:map-conditions}
\Cref{thm:strong-4dvar} is an exact MAP statement provided the posterior admits the displayed density on the admissible state space and the optimizer is sought over that support. If the posterior density has several maximizers, the MAP estimator is nonunique. If the posterior is multimodal or strongly skewed, the MAP point can be far from the posterior mean.
\end{remark}

\subsection{Weak-constraint 4D-Var}

Weak-constraint 4D-Var permits model error inside the window \citep{zupanski1997}. Assume
\begin{align}
X_0 &\sim \mathcal N(x_b,B), \\
X_{t+1} &= \mathcal M_t(X_t) + \xi_t, \qquad \xi_t\sim \mathcal N(0,Q_t), \qquad t=0,\dots,T-1, \\
Y_t &= \mathcal H_t(X_t) + \eta_t, \qquad \eta_t\sim \mathcal N(0,R_t), \qquad t=0,\dots,T,
\end{align}
with all noises mutually independent and each $Q_t$, $R_t$ positive definite.

\begin{proposition}[Weak-constraint 4D-Var objective as negative log-posterior]
\label{prop:weak-4dvar}
Under the weak-constraint model above, the trajectory posterior satisfies
\begin{equation}
p(\xtraj\mid \ytraj)
\propto
\exp\parens{-J_{\mathrm{wc}}(\xtraj)},
\end{equation}
where
\begin{align}
J_{\mathrm{wc}}(\xtraj)
:={}&
\frac12 (x_0-x_b)\transpose B^{-1}(x_0-x_b)
\nonumber\\
&\quad
+\frac12\sum_{t=0}^{T-1}
\parens{x_{t+1}-\mathcal M_t(x_t)}\transpose
Q_t^{-1}
\parens{x_{t+1}-\mathcal M_t(x_t)}
\nonumber\\
&\quad
+\frac12\sum_{t=0}^{T}
\parens{y_t-\mathcal H_t(x_t)}\transpose
R_t^{-1}
\parens{y_t-\mathcal H_t(x_t)},
\label{eq:weak-4dvar-cost}
\end{align}
Hence weak-constraint 4D-Var computes a MAP trajectory under the stated Gaussian assumptions.
\end{proposition}

\begin{proof}
The trajectory prior density with respect to Lebesgue measure on $\R^{n(T+1)}$ is
\[
p(\xtraj)
=
p(x_0)\prod_{t=0}^{T-1}p(x_{t+1}\mid x_t),
\]
where
\begin{align*}
p(x_0)
&\propto
\exp\!\left(
-\frac12 (x_0-x_b)\transpose B^{-1}(x_0-x_b)
\right),\\
p(x_{t+1}\mid x_t)
&\propto
\exp\!\left(
-\frac12
\parens{x_{t+1}-\mathcal M_t(x_t)}\transpose
Q_t^{-1}
\parens{x_{t+1}-\mathcal M_t(x_t)}
\right).
\end{align*}
The likelihood factorizes as
\[
p(\ytraj\mid \xtraj)
=
\prod_{t=0}^{T}p(y_t\mid x_t),
\]
with
\[
p(y_t\mid x_t)
\propto
\exp\!\left(
-\frac12
\parens{y_t-\mathcal H_t(x_t)}\transpose
R_t^{-1}
\parens{y_t-\mathcal H_t(x_t)}
\right).
\]
Multiplying prior and likelihood and taking minus logarithms yields
\[
-\log p(\xtraj\mid \ytraj)
=
J_{\mathrm{wc}}(\xtraj) + C(\ytraj),
\]
where $C(\ytraj)$ is independent of $\xtraj$. Hence the posterior density is proportional to
\[
\exp(-J_{\mathrm{wc}}(\xtraj)).
\]
Therefore any minimizer of $J_{\mathrm{wc}}$ is a MAP trajectory.
\end{proof}

The variational and adjoint formulation of data assimilation is classical in the meteorological literature \citep{ledimet1986,talagrandcourtier1987}, and the negative-log-posterior / MAP interpretation under Gaussian assumptions is explicit in the Bayesian data-assimilation literature \citep{lorenc1986,zupanski1997}. What matters for the present paper is not merely that 4D-Var is Bayesian in origin, but that it solves a different optimization problem from filtering over laws. The optimizer in \cref{thm:strong-4dvar,prop:weak-4dvar} is a state or trajectory; the optimizers in the next section will be full probability distributions.

\subsection{MAP, posterior mean, and posterior law}

Let $\pi$ denote any posterior density on $\R^d$. Then:
\begin{enumerate}[label=\arabic*.]
\item A \emph{MAP estimator} is any maximizer of $\pi(x)$.
\item The \emph{posterior mean} is $\int x\,\pi(x)\dd x$, when the integral exists.
\item The \emph{posterior law} is the full measure with density $\pi$.
\end{enumerate}
Only in special cases, such as symmetric unimodal Gaussian posteriors, do these three objects align in an obvious way. In nonlinear or non-Gaussian data assimilation they typically do not. The difference is essential for later claims about reinforcement learning: a reward that targets posterior means or trajectory modes is not automatically a reward that targets the posterior law.

%% file: 04_posterior_variational.tex
\section{Posterior-as-variational identities: one-step filtering and path-space forms}
\label{sec:variational-posterior}

We now state the central variational identities. They are direct consequences of Bayes' rule, but they deserve to be singled out because they put posterior inference, MAP estimation, Gaussian variational approximation, and KL-regularized control under the same umbrella. The key point for the sequel is that the identity itself is unconditional on the admissible family beyond the stated integrability assumptions, whereas the corresponding minimizer statement needs the target posterior law to belong to that family.

The exact infimum statements below all use the same admissible-truncation device.

\begin{lemma}[Admissible truncations recover the Gibbs infimum]
\label{lem:admissible-truncation}
Let $\mu$ be a probability measure on a measurable space, let $f\ge 0$ be measurable with
\[
0<Z:=\int f\,\dd\mu<\infty,
\]
and define
\[
\nu(\dd x):=\frac{f(x)\,\mu(\dd x)}{Z}.
\]
Let $g\ge 0$ be measurable and finite $\nu$-almost surely, and for each $n\in\mathbb N$ define
\[
A_n:=\set{x:\ e^{-n}\le f(x)\le e^n,\ g(x)\le n},
\qquad
\nu_n(\dd x):=\nu(\dd x\mid A_n).
\]
Then:
\begin{enumerate}[label=\roman*.]
\item $\nu(A_n)\to 1$ and hence $\nu_n$ is well defined for all sufficiently large $n$;
\item $\nu_n\ll\mu$;
\item $\E_{\nu_n}\!\bracks{\abs{\log f}}\le n$ and $\E_{\nu_n}[g]\le n$;
\item $\KL(\nu_n\|\mu)<\infty$;
\item $\KL(\nu_n\|\nu)=-\log \nu(A_n)\to 0$.
\end{enumerate}
Consequently, for any admissible class $\mc A$ that contains all sufficiently large $\nu_n$ and on which
\[
J(q):=\E_q[-\log f]+\KL(q\|\mu)
\]
is well defined, one has
\[
\inf_{q\in\mc A}J(q)=-\log Z.
\]
\end{lemma}

\begin{proof}
Because $\nu$ has density $f/Z$ with respect to $\mu$, it is supported on the set where $0<f<\infty$. Since $g<\infty$ $\nu$-almost surely, the sets $A_n$ increase to a $\nu$-full set. This proves $\nu(A_n)\to 1$.

By construction, $\nu_n\ll\nu\ll\mu$, so \textup{(ii)} holds. On $A_n$ one has $\abs{\log f}\le n$ and $g\le n$, so \textup{(iii)} follows immediately. Moreover,
\[
\frac{\dd \nu_n}{\dd \mu}(x)
=
\frac{f(x)\ind_{A_n}(x)}{Z\,\nu(A_n)},
\]
hence
\[
\log\frac{\dd \nu_n}{\dd \mu}
=
\log f-\log Z-\log \nu(A_n)
\qquad
\nu_n\text{-a.s.}
\]
The right-hand side is $\nu_n$-integrable because $\abs{\log f}\le n$ on $A_n$, so \textup{(iv)} follows. Likewise,
\[
\frac{\dd \nu_n}{\dd \nu}
=
\frac{\ind_{A_n}}{\nu(A_n)},
\]
which gives
\[
\KL(\nu_n\|\nu)
=
\int \log\frac{\dd \nu_n}{\dd \nu}\,\dd \nu_n
=
-\log \nu(A_n)\to 0.
\]
This proves \textup{(v)}.

For any admissible $q\in\mc A$, the Gibbs identity gives
\[
J(q)=\KL(q\|\nu)-\log Z\ge -\log Z.
\]
Applying that same identity to $\nu_n$ yields
\[
J(\nu_n)=\KL(\nu_n\|\nu)-\log Z=-\log \nu(A_n)-\log Z\to -\log Z,
\]
so the lower bound is sharp and therefore $\inf_{q\in\mc A}J(q)=-\log Z$.
\end{proof}

\subsection{One-step filtering identity}

Fix $t\in\set{0,\dots,T}$ and the realized observation $y_t$. Let $p_t^f$ be the forecast density and define the analysis density
\begin{equation}
p_t^a(x_t)
:=
\frac{p(y_t\mid x_t)p_t^f(x_t)}{p(y_t\mid y_{0:t-1})}.
\label{eq:analysis-density-bayes}
\end{equation}
Let $\mc Q_t$ be the collection of densities $q_t$ such that
\begin{enumerate}[label=\roman*.]
\item $q_t\ll p_t^f$,
\item $\E_{q_t}\!\bracks{\abs{\log p(y_t\mid X_t)}}<\infty$,
\item $\KL(q_t\|p_t^f)<\infty$.
\end{enumerate}
For $q_t\in\mc Q_t$, define
\begin{equation}
\mc J_t(q_t)
:=
\E_{q_t}\!\bracks{-\log p(y_t\mid X_t)}
\;+\;
\KL(q_t\|p_t^f).
\label{eq:one-step-functional}
\end{equation}

\begin{theorem}[One-step posterior-as-variational identity]
\label{thm:one-step-variational}
Under \cref{ass:ssm}, for every $q_t\in\mc Q_t$,
\begin{equation}
\mc J_t(q_t)
=
\KL(q_t\|p_t^a)-\log p(y_t\mid y_{0:t-1}).
\label{eq:one-step-identity}
\end{equation}
\begin{equation}
\inf_{q_t\in\mc Q_t}\mc J_t(q_t)
=
-\log p(y_t\mid y_{0:t-1}).
\label{eq:one-step-infimum}
\end{equation}
If, in addition, $p_t^a\in\mc Q_t$, then $\mc J_t$ has the unique minimizer
\begin{equation}
q_t^\star = p_t^a.
\end{equation}
\end{theorem}

\begin{proof}
Because $0<Z_t<\infty$ by \cref{ass:ssm} and $\abs{\log p(y_t\mid X_t)}$ is $q_t$-integrable, one has $p(y_t\mid X_t)\in(0,\infty)$ for $q_t$-almost every $X_t$. Since $q_t\ll p_t^f$, it follows from \cref{eq:analysis-density-bayes} that $q_t\ll p_t^a$, so every logarithm below is well defined $q_t$-almost surely.
By \cref{eq:analysis-density-bayes},
\[
\log p_t^a(x_t)
=
\log p(y_t\mid x_t) + \log p_t^f(x_t) - \log p(y_t\mid y_{0:t-1}).
\]
Rearranging,
\[
-\log p(y_t\mid x_t) + \log \frac{q_t(x_t)}{p_t^f(x_t)}
=
\log \frac{q_t(x_t)}{p_t^a(x_t)} - \log p(y_t\mid y_{0:t-1}).
\]
Integrating with respect to $q_t$ gives
\[
\E_{q_t}\!\bracks{-\log p(y_t\mid X_t)}
\;+\;
\E_{q_t}\!\left[\log \frac{q_t(X_t)}{p_t^f(X_t)}\right]
=
\E_{q_t}\!\left[\log \frac{q_t(X_t)}{p_t^a(X_t)}\right]
-\log p(y_t\mid y_{0:t-1}),
\]
which is precisely \cref{eq:one-step-identity}. Nonnegativity of KL divergence yields
\[
\mc J_t(q_t)\ge -\log p(y_t\mid y_{0:t-1}),
\]
so it remains to prove that this lower bound is sharp. Apply \cref{lem:admissible-truncation} with
\[
\mu(\dd x_t)=p_t^f(x_t)\dd x_t,
\qquad
f(x_t)=p(y_t\mid x_t),
\qquad
g(x_t)=\abs{\log p(y_t\mid x_t)}.
\]
Then $\nu(\dd x_t)=p_t^a(x_t)\dd x_t$, and the truncated posteriors $\nu_n$ from the lemma belong to $\mc Q_t$. Using \cref{eq:one-step-identity},
\[
\mc J_t(\nu_n)
=
\KL(\nu_n\|p_t^a)-\log p(y_t\mid y_{0:t-1})
=
-\log \nu(A_n)-\log p(y_t\mid y_{0:t-1}),
\]
which converges to $-\log p(y_t\mid y_{0:t-1})$. This proves \cref{eq:one-step-infimum}. If $p_t^a\in\mc Q_t$, then evaluating \cref{eq:one-step-identity} at $q_t=p_t^a$ shows that the lower bound is attained there. Uniqueness follows because $\KL(q_t\|p_t^a)=0$ if and only if $q_t=p_t^a$ almost everywhere.
\end{proof}

\begin{remark}[Posterior admissibility is an extra hypothesis]
\label{rem:one-step-admissibility}
The normalization condition $0<Z_t<\infty$ ensures that $p_t^a$ is a well-defined posterior density, but it does not by itself imply $p_t^a\in\mc Q_t$. The additional membership requires the posterior expectation of $\abs{\log p(y_t\mid X_t)}$ and the relative entropy $\KL(p_t^a\|p_t^f)$ to be finite. We therefore separate the exact identity from the minimizer statement.
\end{remark}

\begin{corollary}[Restricted variational families]
\label{cor:one-step-restricted}
Let $\mc F_t\subset\mc Q_t$ be any nonempty admissible family. Then
\begin{equation}
\argmin_{q_t\in \mc F_t} \mc J_t(q_t)
=
\argmin_{q_t\in \mc F_t} \KL(q_t\|p_t^a).
\end{equation}
Thus the two restricted minimization problems are equivalent. In particular, any minimizer of \cref{eq:one-step-functional} over $\mc F_t$ is a reverse-KL projection of the analysis posterior onto that family.
\end{corollary}

\begin{proof}
The additive constant $-\log p(y_t\mid y_{0:t-1})$ in \cref{eq:one-step-identity} does not depend on $q_t$.
\end{proof}

\subsection{Path-space identity}

Let $\mc Q_{\mathrm{path}}$ be the set of densities $q$ on $\R^{n(T+1)}$ such that
\begin{enumerate}[label=\roman*.]
\item $q\ll p(\xtraj)$,
\item $\E_q\!\bracks{\sum_{t=0}^{T}\abs{\log p(y_t\mid X_t)}}<\infty$,
\item $\KL(q\|p(\xtraj))<\infty$.
\end{enumerate}
For $q\in\mc Q_{\mathrm{path}}$, define
\begin{equation}
\mc J_{\mathrm{path}}(q)
:=
\E_q\!\bracks{-\sum_{t=0}^{T}\log p(y_t\mid X_t)}
\;+\;
\KL(q\|p(\xtraj)).
\label{eq:path-functional}
\end{equation}

\begin{theorem}[Path-space posterior-as-variational identity]
\label{thm:path-variational}
Under \cref{ass:ssm}, for every $q\in\mc Q_{\mathrm{path}}$,
\begin{equation}
\mc J_{\mathrm{path}}(q)
=
\KL(q\|p(\xtraj\mid \ytraj))-\log p(\ytraj).
\label{eq:path-identity}
\end{equation}
\begin{equation}
\inf_{q\in\mc Q_{\mathrm{path}}}\mc J_{\mathrm{path}}(q)
=
-\log p(\ytraj).
\label{eq:path-infimum}
\end{equation}
If, in addition, $p(\xtraj\mid \ytraj)\in\mc Q_{\mathrm{path}}$, then $\mc J_{\mathrm{path}}$ has the unique minimizer
\begin{equation}
q^\star(\xtraj) = p(\xtraj\mid \ytraj).
\end{equation}
\end{theorem}

\begin{proof}
Because $0<Z_{0:T}<\infty$ by \cref{ass:ssm} and $\sum_{t=0}^{T}\abs{\log p(y_t\mid X_t)}$ is $q$-integrable, one has $p(y_t\mid X_t)\in(0,\infty)$ for every $t$, $q$-almost surely. Hence $q\ll p(\xtraj\mid \ytraj)$ by \cref{eq:path-posterior}.
By \cref{eq:path-posterior},
\[
\log p(\xtraj\mid \ytraj)
=
\log p(\xtraj) + \sum_{t=0}^{T}\log p(y_t\mid x_t) - \log p(\ytraj).
\]
Therefore
\[
-\sum_{t=0}^{T}\log p(y_t\mid x_t) + \log\frac{q(\xtraj)}{p(\xtraj)}
=
\log\frac{q(\xtraj)}{p(\xtraj\mid \ytraj)} - \log p(\ytraj).
\]
Integrating with respect to $q$ yields \cref{eq:path-identity}, and nonnegativity of KL divergence gives the lower bound
\[
\mc J_{\mathrm{path}}(q)\ge -\log p(\ytraj).
\]
To show that this bound is sharp, apply \cref{lem:admissible-truncation} with
\[
\mu(\dd \xtraj)=p(\xtraj)\dd \xtraj,
\qquad
f(\xtraj)=\prod_{t=0}^{T}p(y_t\mid x_t),
\qquad
g(\xtraj)=\sum_{t=0}^{T}\abs{\log p(y_t\mid x_t)}.
\]
Then $\nu(\dd \xtraj)=p(\xtraj\mid \ytraj)\dd \xtraj$, and the truncated posteriors $\nu_n$ from the lemma belong to $\mc Q_{\mathrm{path}}$. Using \cref{eq:path-identity},
\[
\mc J_{\mathrm{path}}(\nu_n)
=
\KL(\nu_n\|p(\xtraj\mid \ytraj))-\log p(\ytraj)
=
-\log \nu(A_n)-\log p(\ytraj)\to -\log p(\ytraj).
\]
This proves \cref{eq:path-infimum}. If $p(\xtraj\mid \ytraj)\in\mc Q_{\mathrm{path}}$, then evaluating \cref{eq:path-identity} at that posterior shows that the lower bound is attained there. Since $\KL(q\|p(\xtraj\mid \ytraj))=0$ if and only if $q=p(\xtraj\mid \ytraj)$ almost everywhere, the minimizer is unique whenever it exists inside the admissible family.
\end{proof}

\begin{remark}[Path-space posterior admissibility]
\label{rem:path-admissibility}
The same caveat applies on path space: the evidence condition $0<Z_{0:T}<\infty$ makes the smoothing posterior well defined, but the minimizer statement additionally needs $p(\xtraj\mid \ytraj)\in\mc Q_{\mathrm{path}}$.
\end{remark}

\begin{corollary}[Unrestricted, Gaussian, and other restricted families]
\label{cor:path-restricted}
Let $\mc F\subset\mc Q_{\mathrm{path}}$ be any nonempty admissible family. Then
\begin{equation}
\argmin_{q\in\mc F}\mc J_{\mathrm{path}}(q)
=
\argmin_{q\in\mc F}\KL(q\|p(\xtraj\mid \ytraj)).
\end{equation}
In particular:
\begin{enumerate}[label=\arabic*.]
\item if $\mc F=\mc Q_{\mathrm{path}}$ and $p(\xtraj\mid \ytraj)\in\mc Q_{\mathrm{path}}$, then the unique minimizer is the full smoothing posterior;
\item if $\mc F$ is a Gaussian family, then any minimizer is a Gaussian variational approximation, equivalently a reverse-KL projection, of the smoothing posterior.
\end{enumerate}
\end{corollary}

\begin{remark}[MAP and the ``Dirac restriction'' heuristic]
\label{rem:dirac-map}
In a discrete state space, Dirac laws that are supported on prior atoms are admissible, and minimizing \cref{eq:path-functional} over Dirac masses literally selects a path maximizing the posterior mass. In continuous state spaces, however, a Dirac law is singular with respect to a smooth prior density, so the KL term in \cref{eq:path-functional} is not finite. The mathematically correct continuous-state statement is that MAP is recovered by minimizing the negative log-posterior density directly, as in \cref{sec:4dvar}, or as the zero-variance limit of concentrated admissible families. Thus the phrase ``restricting $q$ to Dirac masses gives MAP'' should be read as a shorthand for a limiting point-mass approximation, not as a literal application of \cref{thm:path-variational} with singular measures.
\end{remark}

\begin{remark}[Posterior mean versus variational optimum]
\label{rem:posterior-mean}
The functionals \cref{eq:one-step-functional,eq:path-functional} optimize over \emph{laws}. Their optimizer is the posterior law, not a point estimator. Posterior means enter only after one chooses a decision rule and a loss function, such as quadratic loss. This distinction will matter again in \cref{sec:linear-gaussian,sec:hammoud}.
\end{remark}

The rest of the paper keeps four optimization domains separate. First, \cref{thm:one-step-variational,thm:path-variational} concern unrestricted optimization over admissible laws. Second, restricting the admissible family produces variational approximations such as Gaussian projections. Third, Section~\ref{sec:rl-control} restricts further to \emph{reachable} laws induced by a policy class under fixed passive dynamics, which adds a nontrivial representability constraint. Fourth, Section~\ref{sec:enkf} discusses empirical ensemble approximations, which are random particle or moment surrogates rather than exact optimizers over a law family. Keeping these four levels distinct is what prevents posterior laws, MAP points, Gaussian surrogates, reachable policy laws, and finite ensembles from being conflated.

%% file: 05_rl_control.tex
\section{KL-regularized control and control-as-inference}
\label{sec:rl-control}

The variational identities of \cref{sec:variational-posterior} can be read as control problems. This section separates two levels that are often merged in the literature: unrestricted optimization over path laws relative to passive dynamics, and optimization over the smaller reachable family induced by a policy class with fixed transition kernels. The distinction is precisely where posterior recovery can fail.

\subsection{Posterior as an optimal controlled law}

Let $p_0(\xtraj)$ be the path density induced by passive dynamics on trajectories, and let $\ell_t:\R^n\to\R$ be measurable stage costs. Assume the normalizing constant
\begin{equation}
Z_\ell := \int p_0(\xtraj)\exp\!\left(-\sum_{t=0}^{T}\ell_t(x_t)\right)\dd \xtraj
\end{equation}
is finite and strictly positive. Let $\mc Q_\ell$ be the set of densities $q$ such that
\begin{enumerate}[label=\roman*.]
\item $q\ll p_0$,
\item $\E_q\!\bracks{\sum_{t=0}^{T}\abs{\ell_t(X_t)}}<\infty$,
\item $\KL(q\|p_0)<\infty$.
\end{enumerate}
For any $q\in\mc Q_\ell$, define
\begin{equation}
\mc J_{\ell}(q)
:=
\E_q\!\bracks{\sum_{t=0}^{T}\ell_t(X_t)}
\;+\;
\KL(q\|p_0).
\label{eq:gibbs-functional}
\end{equation}

\begin{theorem}[Gibbs variational identity]
\label{thm:gibbs-control}
For every $q\in\mc Q_\ell$,
\begin{equation}
\mc J_{\ell}(q)
=
\KL(q\|q^\star)-\log Z_\ell,
\label{eq:gibbs-identity}
\end{equation}
where
\begin{equation}
q^\star(\xtraj)
:=
\frac{1}{Z_\ell}
p_0(\xtraj)\exp\!\left(-\sum_{t=0}^{T}\ell_t(x_t)\right).
\label{eq:gibbs-law}
\end{equation}
\begin{equation}
\inf_{q\in\mc Q_\ell}\mc J_{\ell}(q)
=
-\log Z_\ell.
\label{eq:gibbs-infimum}
\end{equation}
If, in addition, $q^\star\in\mc Q_\ell$, then $q^\star$ is the unique minimizer of \cref{eq:gibbs-functional}.
\end{theorem}

\begin{proof}
Because each $\ell_t$ is real valued, the density $q^\star$ from \cref{eq:gibbs-law} is strictly positive wherever $p_0$ is strictly positive. Hence every $q\in\mc Q_\ell$ is absolutely continuous with respect to $q^\star$ as well.
The density $q^\star$ defined in \cref{eq:gibbs-law} satisfies
\[
\log q^\star(\xtraj)
=
\log p_0(\xtraj) - \sum_{t=0}^{T}\ell_t(x_t)-\log Z_\ell.
\]
Rearranging and integrating against $q$ yields \cref{eq:gibbs-identity}, and nonnegativity of KL divergence gives the lower bound
\[
\mc J_\ell(q)\ge -\log Z_\ell.
\]
To show that this bound is sharp while preserving the admissibility conditions in $\mc Q_\ell$, apply \cref{lem:admissible-truncation} with
\[
\mu(\dd \xtraj)=p_0(\xtraj)\dd \xtraj,
\qquad
f(\xtraj)=\exp\!\left(-\sum_{t=0}^{T}\ell_t(x_t)\right),
\qquad
g(\xtraj)=\sum_{t=0}^{T}\abs{\ell_t(x_t)}.
\]
Then $\nu(\dd \xtraj)=q^\star(\xtraj)\dd \xtraj$, and the truncated Gibbs laws $\nu_n$ from the lemma belong to $\mc Q_\ell$. Using \cref{eq:gibbs-identity},
\[
\mc J_\ell(\nu_n)
=
\KL(\nu_n\|q^\star)-\log Z_\ell
=
-\log \nu(A_n)-\log Z_\ell\to -\log Z_\ell.
\]
This proves \cref{eq:gibbs-infimum}. If $q^\star\in\mc Q_\ell$, then evaluating \cref{eq:gibbs-identity} at $q=q^\star$ shows that the lower bound is attained there, and uniqueness follows from strict positivity of KL divergence away from the reference law.
\end{proof}

\begin{remark}[Gibbs admissibility is not automatic]
\label{rem:gibbs-admissibility}
The condition $0<Z_\ell<\infty$ guarantees that \cref{eq:gibbs-law} defines a probability density, but it does not by itself imply $q^\star\in\mc Q_\ell$. The minimizer statement therefore needs the additional integrability and finite-KL membership of the Gibbs law itself.
\end{remark}

\begin{corollary}[Bayesian posterior as the Gibbs law]
\label{cor:posterior-gibbs}
If $p_0(\xtraj)=p(\xtraj)$ is the prior path density from \cref{eq:path-prior} and
\begin{equation}
\ell_t(x_t) = -\log p(y_t\mid x_t),
\end{equation}
then \cref{eq:gibbs-law} becomes
\begin{equation}
q^\star(\xtraj) = p(\xtraj\mid \ytraj).
\end{equation}
Consequently,
\[
\inf_{q\in\mc Q_\ell}\mc J_\ell(q) = -\log p(\ytraj),
\]
and if $p(\xtraj\mid \ytraj)\in\mc Q_\ell$, then Bayesian smoothing is exactly the unique minimizer of \cref{eq:gibbs-functional}.
\end{corollary}

\subsection{Policies, passive dynamics, and relative-entropy control}

To connect \cref{thm:gibbs-control} to reinforcement learning, introduce actions $A_t\in\R^r$ and passive state-action dynamics
\begin{equation}
p_0(\xtraj,\atraj)
=
\rho_0(x_0)\prod_{t=0}^{T-1}\mu_t(a_t\mid x_t)P_t(x_{t+1}\mid x_t,a_t),
\label{eq:passive-state-action-law}
\end{equation}
where $\rho_0$ is the initial density, $\mu_t(\cdot\mid x_t)$ is a passive action kernel, and $P_t(\cdot\mid x_t,a_t)$ is the state transition density. A Markov policy $\pi=\set{\pi_t}_{t=0}^{T-1}$ with $\pi_t(\cdot\mid x)\ll \mu_t(\cdot\mid x)$ induces the state-action path law
\begin{equation}
\qpi(\xtraj,\atraj)
=
\rho_0(x_0)\prod_{t=0}^{T-1}\pi_t(a_t\mid x_t)P_t(x_{t+1}\mid x_t,a_t).
\label{eq:policy-path-law}
\end{equation}
We write
\begin{equation}
\qpi^X(\xtraj)
:=
\int \qpi(\xtraj,\atraj)\dd \atraj
\end{equation}
for the induced state-path marginal, and similarly $p_0^X$ for the state marginal of the passive state-action law \cref{eq:passive-state-action-law}.

At this point three optimization domains must be kept distinct. One may optimize over all admissible state-action laws $q\ll p_0$ on the reference space \cref{eq:passive-state-action-law}; one may optimize only over the reachable subset $\set{\qpi:\pi\text{ admissible}}$ induced by policies with the fixed transition kernel $P_t$; and one may look only at the induced state-path marginals $\set{\qpi^X}$. \Cref{thm:gibbs-control} applies directly to the first domain, while posterior recovery by policies concerns the second and third domains and therefore needs an additional representability argument.

\begin{lemma}[Path-space KL decomposition]
\label{lem:policy-kl}
Assume $\qpi\ll p_0$. Then
\begin{equation}
\KL(\qpi\|p_0)
=
\E_{\qpi}\!\left[\sum_{t=0}^{T-1}\log\frac{\pi_t(A_t\mid X_t)}{\mu_t(A_t\mid X_t)}\right]
=
\sum_{t=0}^{T-1}\E_{\qpi}\!\bracks{\KL\parens{\pi_t(\cdot\mid X_t)\,\|\,\mu_t(\cdot\mid X_t)}}.
\label{eq:policy-kl-decomp}
\end{equation}
\end{lemma}

\begin{proof}
Dividing \cref{eq:policy-path-law} by \cref{eq:passive-state-action-law} yields
\[
\frac{\qpi(\xtraj,\atraj)}{p_0(\xtraj,\atraj)}
=
\prod_{t=0}^{T-1}\frac{\pi_t(a_t\mid x_t)}{\mu_t(a_t\mid x_t)}.
\]
Taking logarithms and expectations under $\qpi$ proves the first identity. The second follows by conditional expectation with respect to $X_t$.
\end{proof}

With the stage cost $\ell_t(X_t)$ acting on states, define the admissible class on the state-action path space by
\begin{equation}
\mc Q_\ell^{\mathrm{SA}}
:=
\left\{
q:\ q\ll p_0,\;
\E_q\!\bracks{\sum_{t=0}^{T}\abs{\ell_t(X_t)}}<\infty,\;
\KL(q\|p_0)<\infty
\right\},
\label{eq:state-action-admissible-class}
\end{equation}
where $p_0$ now denotes the passive state-action law \cref{eq:passive-state-action-law}. This is the admissible class obtained by applying \cref{thm:gibbs-control} on the extended measurable space $(\xtraj,\atraj)$ with cost depending only on the state coordinates. For any $q\in\mc Q_\ell^{\mathrm{SA}}$, define the state-action-law functional
\begin{equation}
\mc J_\ell^{\mathrm{SA}}(q)
:=
\E_q\!\bracks{\sum_{t=0}^{T}\ell_t(X_t)}
\;+\;
\KL(q\|p_0),
\label{eq:state-action-functional}
\end{equation}
Then \cref{thm:gibbs-control,lem:policy-kl} show that applying the Gibbs variational identity on this state-action path space makes the policy optimization problem
\begin{equation}
\inf_{\pi}
\left\{
\E_{\qpi}\!\bracks{\sum_{t=0}^{T}\ell_t(X_t)}
+
\E_{\qpi}\!\left[\sum_{t=0}^{T-1}\log\frac{\pi_t(A_t\mid X_t)}{\mu_t(A_t\mid X_t)}\right]
	\right\}
	\label{eq:policy-objective}
\end{equation}
is equivalent to minimizing $\mc J_\ell^{\mathrm{SA}}(q)$ over the reachable family of policy-induced state-action path laws. This is the relative-entropy control or KL-control form underlying control-as-inference formulations \citep{toussaintstorkey2006,todorov2006,todorov2008,todorov2009,kappen2012,rawlik2012,levine2018}. The unrestricted Gibbs law from \cref{thm:gibbs-control} is therefore the benchmark on the full state-action path space, while the actual policy problem is its restriction to the reachable subset generated by the fixed kernels $P_t$.

\begin{corollary}[Conditional posterior recovery under representability]
\label{cor:policy-posterior}
Assume the passive state marginal satisfies $p_0^X(\xtraj)=p(\xtraj)$, where $p(\xtraj)$ is the Bayesian prior path density from \cref{eq:path-prior}, and let $\ell_t(x_t)=-\log p(y_t\mid x_t)$. Assume in addition that the unrestricted Gibbs law $q^\star$ obtained by applying \cref{thm:gibbs-control} on the state-action path space belongs to $\mc Q_\ell^{\mathrm{SA}}$, so that it is the unique optimizer of the unrestricted state-action problem. If that law is induced by some admissible policy $\pi^\star$, a restrictive representability assumption, then any optimizer of \cref{eq:policy-objective} induces exactly that law, and its state marginal is the Bayesian smoothing posterior:
\begin{equation}
q_{\pi^\star}^{X}(\xtraj) = p(\xtraj\mid \ytraj).
\end{equation}
If the policy class is restricted, the optimal induced state-action law is the reverse-KL projection of $q^\star$ onto the reachable family $\set{q_\pi:\pi\text{ admissible}}$; its state marginal need not be the reverse-KL projection of the smoothing posterior in state space.
\end{corollary}

\begin{proof}
By \cref{thm:gibbs-control} applied on the state-action path space and the hypothesis $q^\star\in\mc Q_\ell^{\mathrm{SA}}$, the unique optimizer over all admissible state-action laws is
\[
q^\star(\xtraj,\atraj)
\propto
p_0(\xtraj,\atraj)\prod_{t=0}^{T}p(y_t\mid x_t).
\]
If $q^\star=q_{\pi^\star}$ for some admissible policy $\pi^\star$, then $\pi^\star$ is optimal for \cref{eq:policy-objective} because that objective is just $\mc J_\ell^{\mathrm{SA}}$ restricted to the reachable family of state-action laws. Integrating out the actions gives
\[
q_{\pi^\star}^{X}(\xtraj)
\propto
p_0^X(\xtraj)\prod_{t=0}^{T}p(y_t\mid x_t)
=
p(\xtraj)\prod_{t=0}^{T}p(y_t\mid x_t),
\]
which is exactly the smoothing posterior by \cref{eq:path-posterior}. The restricted-family statement is the corresponding application of \cref{thm:gibbs-control} on the state-action path space.
\end{proof}

\begin{remark}[Why the representability assumption is nontrivial]
\label{rem:policy-representability}
The condition in \cref{cor:policy-posterior} is far from automatic in the present action-KL model. The unrestricted Gibbs law on state-action paths generally twists the effective one-step state evolution after conditioning on future costs, whereas the admissible policy class keeps the transition kernel $P_t(\cdot\mid x,a)$ fixed and can only reweight the action kernel. Exact posterior recovery by policies therefore requires additional structure beyond matching the likelihood cost and passive state marginal. It can hold in special cases, for example under deterministic dynamics or when the action variable parametrizes the needed transition twist, but not in general.
\end{remark}

\subsection{A precise one-step RL theorem}

We now formulate the filtering version requested in the introduction. At time $t$, let the RL ``action'' be the choice of an admissible analysis density $q_t\in\mc Q_t$ as in \cref{sec:variational-posterior}. Define the reward
\begin{equation}
\mc R_t(q_t)
:=
-\mc J_t(q_t)
=
-\E_{q_t}\!\bracks{-\log p(y_t\mid X_t)}
-\KL(q_t\|p_t^f).
\label{eq:one-step-rl-reward}
\end{equation}

\begin{theorem}[RL solution equals posterior solution in one-step filtering]
\label{thm:rl-one-step}
Under the assumptions of \cref{thm:one-step-variational},
\begin{equation}
\sup_{q_t\in\mc Q_t}\mc R_t(q_t)=\log p(y_t\mid y_{0:t-1}).
\end{equation}
If, in addition, $p_t^a\in\mc Q_t$, then the reward functional \cref{eq:one-step-rl-reward} has the unique maximizer
\begin{equation}
q_t^\star = p_t^a.
\end{equation}
More generally, if actions are parameterized by a family of transport maps or policies whose induced laws form a set $\mc F_t\subset\mc Q_t$, then any optimal action induces a law solving
\begin{equation}
\argmin_{q_t\in \mc F_t}\KL(q_t\|p_t^a).
\end{equation}
In particular, if $p_t^a\in\mc F_t$, then the optimal induced law is exactly $p_t^a$.
\end{theorem}

\begin{proof}
The identity $\mc R_t(q_t)=-\mc J_t(q_t)$ and \cref{thm:one-step-variational} give the supremum formula. If $p_t^a\in\mc Q_t$, the unique-maximizer statement is equivalent to the corresponding minimizer statement for $\mc J_t$. The restricted-family claim follows from \cref{cor:one-step-restricted}.
\end{proof}

\begin{remark}[Why RMSE rewards do not imply posterior recovery]
\label{rem:rmse-insufficient}
\Cref{thm:rl-one-step} is exact because the reward is the negative of the likelihood-plus-KL objective. If the reward is replaced by a generic error metric, such as Euclidean RMSE, then the optimizer is generally a different Gibbs law or a different point-estimation problem. That distinction will be central in \cref{sec:hammoud}.
\end{remark}

\subsection{Bellman recursion for the action-KL model}

For the policy problem \cref{eq:policy-objective}, let $V_t(x)$ denote the optimal cost-to-go starting from state $x$ at time $t$:
\begin{equation}
V_t(x)
:=
\inf_{\pi_{t:T-1}}
\E\!\left[
\sum_{s=t}^{T}\ell_s(X_s)
+
\sum_{s=t}^{T-1}\log\frac{\pi_s(A_s\mid X_s)}{\mu_s(A_s\mid X_s)}
\middle|\,
X_t=x
\right].
\label{eq:policy-value-function}
\end{equation}
The correct dynamic-programming recursion for the present action-KL model is obtained by minimizing over the action kernel at each state.

\begin{proposition}[Bellman recursion and optimal policy for the action-KL model]
\label{prop:soft-bellman}
Assume $V_{t+1}$ is finite and define
\begin{equation}
G_t(x,a)
:=
\int P_t(x'\mid x,a)V_{t+1}(x')\dd x'.
\label{eq:action-cost-to-go}
\end{equation}
Suppose that, for every $x$ and $t=T-1,\dots,0$,
\begin{equation}
0<
\int \mu_t(a\mid x)e^{-G_t(x,a)}\dd a
<\infty.
\label{eq:action-gibbs-normalizer}
\end{equation}
Then the optimal value function for \cref{eq:policy-objective} satisfies
\begin{equation}
V_t(x)
=
\ell_t(x)
-\log
\int \mu_t(a\mid x)
\exp\!\left(-\int P_t(x'\mid x,a)V_{t+1}(x')\dd x'\right)\dd a,
\label{eq:soft-bellman}
\end{equation}
with terminal condition $V_T(x)=\ell_T(x)$. Moreover,
\begin{equation}
\pi_t^\star(a\mid x)
=
\frac{
\mu_t(a\mid x)\exp\!\left(-\int P_t(x'\mid x,a)V_{t+1}(x')\dd x'\right)
}{
\int \mu_t(b\mid x)\exp\!\left(-\int P_t(x'\mid x,b)V_{t+1}(x')\dd x'\right)\dd b
}.
\label{eq:optimal-policy-soft}
\end{equation}
\end{proposition}

\begin{proof}
The Bellman principle gives
\[
V_T(x)=\ell_T(x)
\]
and, for $t<T$,
\[
V_t(x)
=
\ell_t(x)
+
\inf_{\nu\ll\mu_t(\cdot\mid x)}
\left\{
\KL(\nu\|\mu_t(\cdot\mid x))
+
\int G_t(x,a)\,\nu(\dd a)
\right\},
\]
where $\nu$ ranges over admissible one-step action kernels at the current state. For fixed $x$, define
\[
\nu_t^\star(a\mid x)
:=
\frac{\mu_t(a\mid x)e^{-G_t(x,a)}}{\int \mu_t(b\mid x)e^{-G_t(x,b)}\dd b}.
\]
Then the one-step Gibbs identity on the action space yields
\[
\KL(\nu\|\mu_t(\cdot\mid x))
+
\int G_t(x,a)\,\nu(\dd a)
=
\KL(\nu\|\nu_t^\star(\cdot\mid x))
-\log\int \mu_t(a\mid x)e^{-G_t(x,a)}\dd a.
\]
Taking the infimum over $\nu$ proves \cref{eq:soft-bellman}, and the unique minimizer is precisely $\nu_t^\star(\cdot\mid x)$, which is the displayed optimal policy formula.
\end{proof}

\begin{remark}[Relation to desirability formulas]
\label{rem:desirability-caveat}
\Cref{eq:soft-bellman} is the correct Bellman recursion for the action-KL model defined in \cref{eq:passive-state-action-law,eq:policy-path-law}. It differs from the linear desirability recursion familiar from transition-kernel or linearly solvable control formulations, where the control acts directly on the next-state kernel \citep{todorov2006,todorov2009}. In the present model the expectation of $V_{t+1}$ under $P_t(\cdot\mid x,a)$ appears inside the exponential. By Jensen's inequality, replacing it by $\int P_t(x'\mid x,a)e^{-V_{t+1}(x')}\dd x'$ is generally incorrect except in special cases such as deterministic transitions or constant downstream value.
\end{remark}

\begin{remark}[Temperatures and tempered posteriors]
\label{rem:temperature}
If one replaces \cref{eq:policy-objective} by
\[
\E_{\qpi}\!\bracks{\sum_t \beta\,\ell_t(X_t)}
\;+\;
\alpha\,\KL(\qpi\|p_0)
\]
with $\alpha,\beta>0$, then the optimizer becomes
\[
q^\star(\xtraj,\atraj)
\propto
p_0(\xtraj,\atraj)\exp\!\left(-\frac{\beta}{\alpha}\sum_{t=0}^{T}\ell_t(x_t)\right).
\]
For likelihood costs $\ell_t=-\log p(y_t\mid x_t)$ this is a tempered posterior on the relevant path space. Exact Bayesian recovery therefore requires the reward scaling and KL weight to be calibrated correctly, and in the policy formulation it still requires the corresponding Gibbs law to be reachable by the policy class.
\end{remark}

%% file: 06_linear_gaussian.tex
\section{Linear-Gaussian specialization: Kalman filter}
\label{sec:linear-gaussian}

We now show how the previous variational principles reduce to the classical Kalman analysis step \citep{kalman1960}. The key point is that in the linear-Gaussian case the one-step posterior is itself Gaussian, so the unrestricted variational minimizer from \cref{thm:one-step-variational} already lies inside the Gaussian family.

\begin{assumption}[Linear-Gaussian forecast and observation model]
\label{ass:linear-gaussian}
At time $t$, let
\begin{align}
X_t &\sim p_t^f = \mathcal N(m_t^f,P_t^f), \\
Y_t &= H_tX_t + \eta_t, \qquad \eta_t\sim \mathcal N(0,R_t),
\end{align}
where $P_t^f\in\R^{n\times n}$ and $R_t\in\R^{m\times m}$ are symmetric positive definite and $\eta_t$ is independent of $X_t$.
\end{assumption}

\begin{theorem}[Kalman analysis posterior]
\label{thm:kalman-posterior}
Under \cref{ass:linear-gaussian}, the analysis posterior is Gaussian:
\begin{equation}
p_t^a = \mathcal N(m_t^a,P_t^a).
\end{equation}
Its information-form parameters are
\begin{align}
\parens{P_t^a}^{-1} &= \parens{P_t^f}^{-1} + H_t\transpose R_t^{-1}H_t, \label{eq:kalman-information-cov}\\
m_t^a &= P_t^a\parens{\parens{P_t^f}^{-1}m_t^f + H_t\transpose R_t^{-1}y_t}. \label{eq:kalman-information-mean}
\end{align}
Equivalently, with Kalman gain
\begin{equation}
K_t := P_t^fH_t\transpose\parens{H_tP_t^fH_t\transpose + R_t}^{-1},
\label{eq:kalman-gain}
\end{equation}
the covariance-form update is
\begin{align}
P_t^a &= \parens{\Id-K_tH_t}P_t^f, \label{eq:kalman-covariance-form}\\
m_t^a &= m_t^f + K_t\parens{y_t-H_tm_t^f}. \label{eq:kalman-mean-form}
\end{align}
\end{theorem}

\begin{proof}
Bayes' rule gives
\[
p_t^a(x_t)
\propto
\exp\!\left(
-\frac12(x_t-m_t^f)\transpose (P_t^f)^{-1}(x_t-m_t^f)
-\frac12(y_t-H_tx_t)\transpose R_t^{-1}(y_t-H_tx_t)
\right).
\]
Expanding the quadratic form in $x_t$ yields
\begin{align*}
&(x_t-m_t^f)\transpose (P_t^f)^{-1}(x_t-m_t^f)
+(y_t-H_tx_t)\transpose R_t^{-1}(y_t-H_tx_t) \\
&\qquad=
x_t\transpose\parens{(P_t^f)^{-1}+H_t\transpose R_t^{-1}H_t}x_t
-2x_t\transpose\parens{(P_t^f)^{-1}m_t^f + H_t\transpose R_t^{-1}y_t}
+c(y_t),
\end{align*}
where $c(y_t)$ does not depend on $x_t$. Since $(P_t^f)^{-1}+H_t\transpose R_t^{-1}H_t$ is symmetric positive definite, completing the square gives \cref{eq:kalman-information-cov,eq:kalman-information-mean}. Applying the Woodbury identity,
\[
\parens{(P_t^f)^{-1}+H_t\transpose R_t^{-1}H_t}^{-1}
=
P_t^f - P_t^fH_t\transpose\parens{H_tP_t^fH_t\transpose+R_t}^{-1}H_tP_t^f,
\]
yields \cref{eq:kalman-covariance-form}. Substituting this identity into \cref{eq:kalman-information-mean} gives \cref{eq:kalman-mean-form}. The details are summarized again in \cref{app:matrix-identities}.
\end{proof}

\subsection{Kalman update from the central variational theorem}

The previous theorem is the direct Bayesian calculation. We now show that the same formulas arise from minimizing the one-step variational functional over Gaussian candidate laws.

\begin{proposition}[Gaussian minimization of the one-step functional]
\label{prop:gaussian-functional}
Under \cref{ass:linear-gaussian}, let $q=\mathcal N(m,P)$ with $P$ symmetric positive definite. Then
\begin{align}
\mc J_t(q)
={}&
\frac12 (H_tm-y_t)\transpose R_t^{-1}(H_tm-y_t)
+\frac12\tr\parens{H_t\transpose R_t^{-1}H_tP}
\nonumber\\
&\quad
+\frac12 (m-m_t^f)\transpose (P_t^f)^{-1}(m-m_t^f)
+\frac12\tr\parens{(P_t^f)^{-1}P}
\nonumber\\
&\quad
-\frac12\log\det P
+\frac12\log\det P_t^f
-\frac n2
+c_t,
\label{eq:gaussian-J}
\end{align}
where $c_t$ depends on $y_t$ and $R_t$ but not on $(m,P)$. The unique minimizer over Gaussian $q$ is the posterior Gaussian from \cref{thm:kalman-posterior}; equivalently, the first-order conditions are exactly \cref{eq:kalman-information-cov,eq:kalman-information-mean}.
\end{proposition}

\begin{proof}
The expectation of the quadratic observation term under $q$ is
\[
\E_q\!\bracks{(y_t-H_tX_t)\transpose R_t^{-1}(y_t-H_tX_t)}
=
(y_t-H_tm)\transpose R_t^{-1}(y_t-H_tm)
+\tr\parens{H_t\transpose R_t^{-1}H_tP}.
\]
The Gaussian KL divergence satisfies
\begin{align*}
\KL\parens{\mathcal N(m,P)\,\|\,\mathcal N(m_t^f,P_t^f)}
= {}&
\frac12\tr\parens{(P_t^f)^{-1}P}
+\frac12(m-m_t^f)\transpose (P_t^f)^{-1}(m-m_t^f)
\nonumber\\
&\quad
-\frac n2
+\frac12\log\det P_t^f
-\frac12\log\det P.
\end{align*}
Substituting these expressions into \cref{eq:one-step-functional} gives \cref{eq:gaussian-J}. Differentiating with respect to $m$ and $P$ yields the stationarity equations
\begin{align*}
\parens{(P_t^f)^{-1}+H_t\transpose R_t^{-1}H_t}m
&=
(P_t^f)^{-1}m_t^f + H_t\transpose R_t^{-1}y_t, \\
P^{-1}
&=
(P_t^f)^{-1}+H_t\transpose R_t^{-1}H_t.
\end{align*}
These are exactly \cref{eq:kalman-information-cov,eq:kalman-information-mean}. Strict convexity of \cref{eq:gaussian-J} in $(m,P)$ gives uniqueness. A fuller derivation appears in \cref{app:gaussian-minimization}.
\end{proof}

\begin{corollary}[The Kalman analysis as a variational optimum]
\label{cor:kalman-variational}
In the linear-Gaussian setting, the analysis posterior from \cref{thm:kalman-posterior} is the unique global minimizer of the one-step functional \cref{eq:one-step-functional}. Since the posterior itself is Gaussian, no approximation enters at the variational level.
\end{corollary}

\subsection{Posterior means and deterministic estimators}

\begin{proposition}[Posterior mean under quadratic loss]
\label{prop:posterior-mean-quadratic}
Let $X$ be any square-integrable random vector and let $\pi(\cdot\mid y)$ denote a posterior law of $X$ given $Y=y$. The unique minimizer of
\[
a\mapsto \int \norm{x-a}^2\,\pi(\dd x\mid y)
\]
is the posterior mean $\int x\,\pi(\dd x\mid y)$. In the linear-Gaussian setting of \cref{ass:linear-gaussian}, this minimizer is $m_t^a$.
\end{proposition}

\begin{proof}
Let $m(y)=\int x\,\pi(\dd x\mid y)$. Then
\[
\int \norm{x-a}^2\,\pi(\dd x\mid y)
=
\int \norm{x-m(y)}^2\,\pi(\dd x\mid y)
\;+\;
\norm{m(y)-a}^2,
\]
because the cross term
\[
2\int \inner{x-m(y)}{m(y)-a}\,\pi(\dd x\mid y)
\]
vanishes by definition of $m(y)$. The second term is uniquely minimized at $a=m(y)$.
\end{proof}

\Cref{prop:posterior-mean-quadratic} will be relevant later: a deterministic RL objective with squared-error loss targets a posterior mean, not the posterior law. In the Gaussian case these coincide with the Kalman analysis mean, but the distinction remains conceptually important.

%% file: 07_enkf.tex
\section{Ensemble Kalman filter as an ensemble approximation}
\label{sec:enkf}

The ensemble Kalman filter is best understood here as a Gaussian and finite-ensemble approximation to the forecast-to-analysis map $p_t^f\mapsto p_t^a$ in the linear-Gaussian setting, together with a moment closure in more general settings \citep{evensen2003,kelly2014}. This interpretation is exact only in specific senses; at finite ensemble size it is not the full Bayesian posterior law on state space.

\subsection{Gaussian surrogate update}

Given a forecast ensemble $\set{x_t^{f,(i)}}_{i=1}^{N}$ intended to approximate $p_t^f$, define the sample mean and covariance
\begin{align}
\hat m_t^f &:= \frac1N\sum_{i=1}^{N} x_t^{f,(i)}, \\
\hat P_t^f &:= \frac{1}{N-1}\sum_{i=1}^{N}\parens{x_t^{f,(i)}-\hat m_t^f}\parens{x_t^{f,(i)}-\hat m_t^f}\transpose.
\end{align}
If one replaces the true forecast law $p_t^f$ by the Gaussian surrogate $\mathcal N(\hat m_t^f,\hat P_t^f)$, then the exact Gaussian posterior update from \cref{thm:kalman-posterior} yields surrogate analysis parameters
\begin{align}
\hat K_t &:= \hat P_t^fH_t\transpose\parens{H_t\hat P_t^fH_t\transpose + R_t}^{-1}, \label{eq:sample-gain}\\
\hat m_t^a &:= \hat m_t^f + \hat K_t\parens{y_t-H_t\hat m_t^f}, \label{eq:sample-mean-update}\\
\hat P_t^a &:= \parens{\Id-\hat K_tH_t}\hat P_t^f. \label{eq:sample-cov-update}
\end{align}
Because $R_t\succ 0$, the matrix $H_t\hat P_t^fH_t\transpose+R_t$ is always symmetric positive definite, so the gain \cref{eq:sample-gain} is well defined even when the sample covariance $\hat P_t^f$ is singular. If $\hat P_t^f\succ 0$, then $\mathcal N(\hat m_t^a,\hat P_t^a)$ is exactly the Gaussian posterior for the surrogate forecast density $\mathcal N(\hat m_t^f,\hat P_t^f)$, equivalently the minimizer of the one-step functional \cref{eq:one-step-functional} over Gaussian densities built from those surrogate moments. If $\hat P_t^f$ is singular, the same mean and covariance formulas still define the EnKF moment update, but the corresponding Gaussian law is degenerate on an affine subspace and therefore lies outside the density-based variational family of \cref{sec:variational-posterior}.

\subsection{Perturbed-observation EnKF}

The stochastic or perturbed-observation EnKF updates ensemble members by
\begin{equation}
x_t^{a,(i)}
=
x_t^{f,(i)}+\hat K_t\parens{y_t+\varepsilon_t^{(i)}-H_tx_t^{f,(i)}},
\qquad
\varepsilon_t^{(i)}\stackrel{\text{i.i.d.}}{\sim}\mathcal N(0,R_t).
\label{eq:stochastic-enkf}
\end{equation}
The observation perturbations were introduced precisely so that the empirical covariance of the analysis ensemble is consistent, in expectation, with the Kalman covariance update \citep{burgers1998}. The method was developed and popularized as a practical large-scale sequential estimator built from Gaussian Bayesian update formulas by Evensen \cite{evensen2003}.

\begin{proposition}[Exact stochastic transport when the gain is exact]
\label{prop:exact-stochastic-transport}
Assume the exact linear-Gaussian setting of \cref{ass:linear-gaussian}. Let $X_t^f\sim\mathcal N(m_t^f,P_t^f)$ and $\varepsilon_t\sim\mathcal N(0,R_t)$ be independent, and define
\begin{equation}
X_t^a := X_t^f + K_t\parens{y_t+\varepsilon_t-H_tX_t^f},
\end{equation}
where $K_t$ is the exact Kalman gain \cref{eq:kalman-gain}. Then $X_t^a\sim \mathcal N(m_t^a,P_t^a)$, with $(m_t^a,P_t^a)$ given by \cref{eq:kalman-covariance-form,eq:kalman-mean-form}.
\end{proposition}

\begin{proof}
The variable $X_t^a$ is an affine image of a jointly Gaussian vector, so it is Gaussian. Its mean is
\[
\E[X_t^a]
=
m_t^f + K_t(y_t-H_tm_t^f)
=
m_t^a.
\]
Its covariance is
\begin{align*}
\Cov(X_t^a)
&=
\parens{\Id-K_tH_t}P_t^f\parens{\Id-K_tH_t}\transpose
+K_tR_tK_t\transpose \\
&=
P_t^f-K_tH_tP_t^f-P_t^fH_t\transpose K_t\transpose
+K_t\parens{H_tP_t^fH_t\transpose+R_t}K_t\transpose \\
&=
P_t^f-K_tH_tP_t^f
=
P_t^a,
\end{align*}
where the final identities use \cref{eq:kalman-gain}.
\end{proof}

\Cref{prop:exact-stochastic-transport} explains the origin of the EnKF analysis map. If one started with i.i.d.\ forecast particles and still used the exact deterministic gain $K_t$, the transformed particles would remain i.i.d.\ draws from the exact analysis posterior. What standard EnKF does instead is replace the exact gain $K_t$ and exact forecast moments by their ensemble estimates. That produces a practical finite-ensemble algorithm, but it is no longer literally the exact posterior law on $\R^n$.

\subsection{Finite ensemble versus infinite ensemble}

Several distinct statements must be kept separate.
\begin{enumerate}[label=\arabic*.]
\item \textbf{Exact Kalman posterior.} In the linear-Gaussian model, the posterior law is exactly $\mathcal N(m_t^a,P_t^a)$.
\item \textbf{Gaussian surrogate update.} If $\hat P_t^f\succ0$, then \cref{eq:sample-gain,eq:sample-mean-update,eq:sample-cov-update} are the exact Gaussian posterior formulas for the surrogate forecast density $\mathcal N(\hat m_t^f,\hat P_t^f)$. If $\hat P_t^f$ is singular, the same formulas still define a moment update, but not a nondegenerate posterior density on $\R^n$.
\item \textbf{Finite-$N$ perturbed-observation EnKF.} With the sample gain $\hat K_t$, the analysis ensemble
\[
\hat \nu_t^{a,N}:=\frac1N\sum_{i=1}^{N}\delta_{x_t^{a,(i)}}
\]
is a random empirical measure. At finite $N$ it is generally neither equal to the exact posterior law nor a sample of independent posterior draws.
\item \textbf{Finite-$N$ deterministic square-root EnKF.} Square-root or transform filters deterministically enforce the surrogate analysis mean and covariance, but the resulting empirical measure is still only an ensemble representation of selected Gaussian moments. Even in the linear-Gaussian setting, this is a finite-ensemble surrogate-law statement, not exact equality with the full posterior law on $\R^n$.
\item \textbf{Infinite-ensemble linear-Gaussian limit.} Under linear-Gaussian assumptions, i.i.d.\ initialization, and standard propagation hypotheses, the standard perturbed-observation EnKF satisfies $\hat m_t^f\to m_t^f$ and $\hat P_t^f\to P_t^f$, hence converges to the Kalman update as $N\to\infty$ \citep{mandel2011}. Deterministic square-root variants are designed to realize the same limiting Kalman moment update, but the exact large-ensemble theorem is variant-dependent and is not needed for the present hierarchy. This linear-Gaussian infinite-ensemble statement is the exactness result usually invoked in the EnKF literature.
\item \textbf{Nonlinear or non-Gaussian settings.} Even with $N\to\infty$, EnKF-type methods retain a Gaussian or moment-closure update structure and therefore generally do not recover the exact nonlinear non-Gaussian posterior \citep{evensen2003,kelly2014,vanleeuwen2019,reich2019}.
\end{enumerate}

\subsection{Deterministic square-root variants}

Deterministic square-root or transform filters avoid perturbed observations. Instead, they update the ensemble mean by the Kalman mean formula and apply a deterministic transform to the ensemble anomalies so that the empirical analysis covariance matches the target surrogate covariance \citep{whitakerhamill2002,tippett2003}. From the perspective of \cref{sec:variational-posterior,sec:linear-gaussian}, these methods implement the same Gaussian surrogate posterior update as the perturbed-observation EnKF, but by a deterministic ensemble transport rather than by randomized observation perturbations.

The exactness caveat remains the same. At finite ensemble size, a square-root EnKF matches the prescribed surrogate moments by construction, but that moment-matching property should not be confused with exact posterior recovery. The empirical measure of the transformed ensemble is not the full Bayesian posterior law on $\R^n$; it is an ensemble representation designed to reproduce selected Gaussian summary statistics, with exact posterior recovery only in the corresponding linear-Gaussian infinite-ensemble limit.

\subsection{Relation to controlled particle methods}
\label{subsec:controlled-particle}

The control viewpoint extends beyond the EnKF. In feedback particle filtering, particles are driven by a feedback control law chosen so that, at the mean-field level, their conditional distribution matches the true filtering distribution \citep{yang2013}. In this sense, the filtering problem is recast as an optimal-control problem on interacting particles. Controlled interacting particle systems include both the feedback particle filter and the EnKF as special cases of feedback-based Bayesian update mechanisms \citep{taghvaeimehta2023}.

This control interpretation provides a natural bridge to the present variational picture. The feedback control law in feedback particle filtering is designed so that the particle law matches the posterior, while the EnKF uses a linear-Gaussian control law determined by first and second moments. Reich's Schr\"odinger and transport perspective makes the same point from another angle: sequential data assimilation can be viewed as transporting a forecast measure to an analysis measure, with optimal-transport and Schr\"odinger-bridge ideas clarifying the relation between exact posterior transport and its Gaussian or ensemble approximations \citep{reich2019}. In this hierarchy, the EnKF is a particular controlled transport with Gaussian structure, not the general posterior transport itself.

%% file: 08_hammoud.tex
\section{RMSE-based RL data assimilation and Hammoud et al.}
\label{sec:hammoud}

The preceding sections provide a criterion for when an RL formulation is posterior-exact: the reward and regularization must combine to realize the likelihood-plus-KL objective from \cref{thm:rl-one-step,cor:policy-posterior}. This section uses that criterion to interpret existing RL-based data assimilation work, especially Hammoud et al.\ \cite{hammoud2024}. The claims here are limited to the objective described in the published paper, not to unreported implementation details.

\subsection{What the published objective supports}

Hammoud et al.\ \cite{hammoud2024} study data assimilation for the Lorenz--63 and Lorenz--96 systems using a stochastic deep RL policy that applies state corrections and generates ensembles of assimilated realizations. As stated in the article abstract, the agent optimizes a geometric return built from terms proportional to the negative root-mean-squared error between observations and corresponding forecast states, and the resulting method performs favorably against the EnKF on the reported test problems.

That is already a meaningful contribution. It shows that RL can learn useful nonlinear correction strategies, can be implemented with stochastic action policies, and can produce Monte Carlo ensembles rather than only single corrected states. It also suggests that RL policies can help in settings where EnKF's Gaussian update structure is restrictive.

\subsection{What RMSE rewards do not prove}

An RMSE-style reward does \emph{not} by itself imply that the learned policy samples from the Bayesian analysis posterior. To see why, consider the one-step entropy-regularized objective, interpreted over laws $q_t\ll p_t^f$ for which the displayed expectation and relative entropy are finite,
\begin{equation}
\inf_{q_t\ll p_t^f}
\left\{
\E_{q_t}\!\bracks{\lambda\,r(X_t;y_t)}
+\KL(q_t\|p_t^f)
\right\},
\label{eq:generic-gibbs-objective}
\end{equation}
for some measurable data-fit penalty $r(\cdot;y_t)$ such that
\[
Z_r:=\int p_t^f(x_t)e^{-\lambda r(x_t;y_t)}\dd x_t\in(0,\infty).
\]
Under the same finiteness logic as in \cref{thm:gibbs-control}, here encoded by $Z_r\in(0,\infty)$ together with finiteness of the displayed objective, the optimizer is
\begin{equation}
q_t^\star(x_t)
\propto
p_t^f(x_t)\exp\parens{-\lambda\,r(x_t;y_t)}.
\label{eq:gibbs-rmse}
\end{equation}
Therefore:
\begin{enumerate}[label=\arabic*.]
\item if $r(x_t;y_t)=-\log p(y_t\mid x_t)$, then $q_t^\star=p_t^a$, the Bayesian posterior;
\item if $r$ is Euclidean RMSE or any other mismatch metric, then $q_t^\star$ is generally a \emph{different} Gibbs law, associated with a pseudo-likelihood proportional to $\exp(-\lambda r)$.
\end{enumerate}
This observation does not diminish the practical value of the learned policy. It only says that posterior equivalence requires a specific objective, and RMSE alone is not that objective.

\begin{remark}[Squared Mahalanobis error is the special Bayesian case]
\label{rem:mahalanobis}
If the reward uses the quadratic Mahalanobis misfit
\[
r(x_t;y_t)=\frac12(y_t-H_tx_t)\transpose R_t^{-1}(y_t-H_tx_t),
\]
then $\exp(-r(x_t;y_t))$ is proportional to the Gaussian likelihood \citep{kalman1960,lorenc1986}. In that case, and only together with the correct KL regularization, the entropy-regularized optimizer coincides with the Bayesian posterior. The issue is therefore not whether an RL reward uses an ``error'' term, but whether that term is exactly the negative log-likelihood associated with the intended observation model.
\end{remark}

\subsection{What would be needed for a posterior-equivalence theorem}

To obtain a theorem of the form ``the RL solution equals the Bayesian analysis posterior,'' one needs at least the following ingredients.
\begin{enumerate}[label=\arabic*.]
\item A forecast or passive law $p_t^f$ or $p_0$ relative to which the controlled law is measured.
\item A reward term equal, up to additive constants, to the observation log-likelihood.
\item A KL or relative-entropy regularizer with the correct temperature, so that the induced Gibbs law is the desired posterior rather than a tempered or pseudo-posterior.
\item A policy or transport family rich enough to represent the target posterior, or an explicit statement that one is computing the reverse-KL projection of the posterior onto a restricted family.
\end{enumerate}
These conditions are exactly those of \cref{thm:rl-one-step,cor:policy-posterior}. Without them, a learned policy may still be an excellent estimator, but it is not automatically an exact posterior solver.

\subsection{Interpretive hierarchy for RL-based data assimilation}

The theory developed here suggests the following hierarchy.
\begin{enumerate}[label=\arabic*.]
\item \textbf{Deterministic squared-error RL.} The target is a point estimator, typically close to a posterior mean under quadratic loss.
\item \textbf{Stochastic RL with generic data-fit reward.} The target is a Gibbs law determined by the chosen reward, not necessarily the Bayesian posterior.
\item \textbf{KL-regularized RL with exact likelihood reward.}
The target is the Bayesian posterior or its reverse-KL projection onto the policy family.
\item \textbf{Gaussian or ensemble-restricted KL-regularized RL.} The target is a Gaussian or ensemble approximation of the posterior, yielding Kalman-like or EnKF-like updates.
\end{enumerate}
Hammoud et al.\ fit naturally into the second category based on the published objective description. Moving to the third category would require the control-as-inference structure to be built into the reward and regularizer explicitly.

%% file: 10_conclusion.tex
\section{Conclusion}
\label{sec:conclusion}

The manuscript isolates a single exact theorem chain behind several overlapping languages in data assimilation. The analysis posterior $p_t^a$ and smoothing posterior $p(\xtraj\mid \ytraj)$ satisfy exact likelihood-plus-KL identities and are the unique minimizers when those posterior laws are admissible; strong- and weak-constraint 4D-Var are MAP procedures under the stated Gaussian assumptions; KL-regularized control recovers the Bayesian posterior only when the passive dynamics, likelihood cost, temperature, and a restrictive representability condition on the policy class are matched correctly; and the Kalman filter and EnKF then appear as the exact linear-Gaussian specialization and its Gaussian or finite-ensemble approximation.

The payoff is precision. It prevents posterior-exact claims from being transferred to objectives that in fact target means, modes, surrogate Gaussians, reachable policy laws, or RMSE-driven pseudo-posteriors. That precision is not restrictive; it gives a clean design principle for new algorithms. Keep the posterior identity explicit at the level of the objective, state the admissibility and representability assumptions needed for exact recovery, and then say exactly which restricted family, transport class, or ensemble approximation is being used to make the problem computationally feasible.

\section*{Acknowledgements}
This research has been supported by the Gordon and Betty Moore Foundation’s Postdoctoral Fellowship. 

%% file: appendix.tex
\section{Supplementary proofs and remarks}

\subsection{Gaussian minimization details}
\label{app:gaussian-minimization}

For completeness, we record the derivative calculation used in \cref{prop:gaussian-functional}. Let
\begin{align*}
F(m,P)
:={}&
\frac12 (Hm-y)\transpose R^{-1}(Hm-y)
+\frac12\tr(H\transpose R^{-1}HP)
\nonumber\\
&\quad
+\frac12(m-m_f)\transpose P_f^{-1}(m-m_f)
+\frac12\tr(P_f^{-1}P)
-\frac12\log\det P.
\end{align*}
Then
\[
\nabla_m F
=
H\transpose R^{-1}(Hm-y)+P_f^{-1}(m-m_f),
\]
so the stationary condition $\nabla_mF=0$ gives
\[
\parens{P_f^{-1}+H\transpose R^{-1}H}m = P_f^{-1}m_f + H\transpose R^{-1}y.
\]
For the covariance variable, standard matrix calculus gives
\[
\nabla_P F
=
\frac12 H\transpose R^{-1}H + \frac12 P_f^{-1} - \frac12 P^{-1},
\]
and the stationary condition is therefore
\[
P^{-1}=P_f^{-1}+H\transpose R^{-1}H.
\]
These are the information-form Kalman equations. Strict convexity in $m$ follows from positive definiteness of $P_f^{-1}+H\transpose R^{-1}H$, while strict convexity in $P$ on the cone of symmetric positive definite matrices follows from the convexity of $\tr(AP)-\log\det P$ when $A$ is positive definite.

\subsection{MAP as a zero-variance limit}
\label{app:map-limit}

This subsection makes precise the heuristic from \cref{rem:dirac-map} under conditions that preserve absolute continuity. Let $\pi$ be a posterior density on $\R^d$ such that $\pi(x)>0$ for every $x\in\R^d$ and $\log \pi$ is continuous with at most quadratic growth:
\[
\abs{\log \pi(z)}\le C(1+\norm{z}^2)
\qquad\text{for all }z\in\R^d
\]
for some constant $C<\infty$. For $x\in\R^d$ and $\varepsilon>0$, let $q_{x,\varepsilon}=\mathcal N(x,\varepsilon \Id)$. Then
\[
\KL(q_{x,\varepsilon}\|\pi)
=
-\E_{q_{x,\varepsilon}}[\log \pi(Z)] - H(q_{x,\varepsilon}),
\]
where $H(q_{x,\varepsilon})$ denotes the differential entropy of $q_{x,\varepsilon}$. Since
\[
H(q_{x,\varepsilon}) = \frac d2\log(2\pi e\varepsilon)
\]
is independent of $x$, minimizing $\KL(q_{x,\varepsilon}\|\pi)$ over $x$ is equivalent to maximizing $\E_{q_{x,\varepsilon}}[\log \pi(Z)]$ over $x$. By continuity of $\log \pi$,
\[
\E_{q_{x,\varepsilon}}[\log \pi(Z)] \to \log \pi(x)
\qquad\text{as }\varepsilon\downarrow 0.
\]
Indeed, if $\Xi\sim\mathcal N(0,\Id)$ then $Z=x+\sqrt{\varepsilon}\,\Xi$, so $\log \pi(Z)\to\log \pi(x)$ almost surely. The quadratic-growth bound gives
\[
\abs{\log \pi(Z)}
\le
C\parens{1+2\norm{x}^2+2\varepsilon\norm{\Xi}^2},
\]
which is integrable for every fixed $x$ and $0<\varepsilon\le 1$. Dominated convergence therefore justifies the displayed limit.

Now define
\[
F_\varepsilon(x):=\E\bracks{\log \pi\parens{x+\sqrt{\varepsilon}\,\Xi}},
\]
so that minimizing $\KL(q_{x,\varepsilon}\|\pi)$ is equivalent to maximizing $F_\varepsilon(x)$. Let $\varepsilon_k\downarrow 0$ and let $x_{\varepsilon_k}$ be maximizers such that $x_{\varepsilon_k}\to x^\star$. Because this subsequence is bounded, there exists $M<\infty$ with $\norm{x_{\varepsilon_k}}\le M$ for all $k$. For every $k$ with $\varepsilon_k\le 1$,
\[
\abs{\log \pi\parens{x_{\varepsilon_k}+\sqrt{\varepsilon_k}\,\Xi}}
\le
C\parens{1+2\norm{x_{\varepsilon_k}}^2+2\varepsilon_k\norm{\Xi}^2}
\le
C\parens{1+2M^2+2\norm{\Xi}^2},
\]
and the right-hand side is integrable. Since
\[
x_{\varepsilon_k}+\sqrt{\varepsilon_k}\,\Xi\to x^\star
\qquad\text{almost surely},
\]
dominated convergence gives
\[
F_{\varepsilon_k}(x_{\varepsilon_k})\to \log \pi(x^\star).
\]
For any fixed $y\in\R^d$, the fixed-center argument above gives
\[
F_{\varepsilon_k}(y)\to \log \pi(y).
\]
Because $x_{\varepsilon_k}$ maximizes $F_{\varepsilon_k}$,
\[
F_{\varepsilon_k}(x_{\varepsilon_k})\ge F_{\varepsilon_k}(y)
\qquad\text{for every }k.
\]
Passing to the limit yields $\log \pi(x^\star)\ge \log \pi(y)$ for every $y$, so $x^\star$ is a MAP point of $\pi$. This is the precise continuous-state version of the common slogan that the point-mass restriction yields MAP. If the posterior has constrained support, the same conclusion requires support-preserving smooth approximants rather than full Gaussians.

\subsection{Matrix identities}
\label{app:matrix-identities}

The covariance-form Kalman update follows from the Woodbury identity
\[
\parens{A^{-1}+U C^{-1}V}^{-1}
=
A-AU\parens{C+VAU}^{-1}VA,
\]
applied with $A=P_t^f$, $U=H_t\transpose$, $C=R_t$, and $V=H_t$. This gives
\[
P_t^a
=
\parens{(P_t^f)^{-1}+H_t\transpose R_t^{-1}H_t}^{-1}
=
P_t^f - P_t^fH_t\transpose\parens{H_tP_t^fH_t\transpose+R_t}^{-1}H_tP_t^f.
\]
If $K_t$ is defined by \cref{eq:kalman-gain}, then
\[
P_t^a = (\Id-K_tH_t)P_t^f.
\]
The equivalent Joseph form,
\[
P_t^a
=
\parens{\Id-K_tH_t}P_t^f\parens{\Id-K_tH_t}\transpose + K_tR_tK_t\transpose,
\]
is obtained by expanding the right-hand side:
\[
\parens{\Id-K_tH_t}P_t^f\parens{\Id-K_tH_t}\transpose + K_tR_tK_t\transpose
=
P_t^f-K_tH_tP_t^f-P_t^fH_t\transpose K_t\transpose
+K_t\parens{H_tP_t^fH_t\transpose+R_t}K_t\transpose,
\]
and then using
\[
K_t\parens{H_tP_t^fH_t\transpose+R_t}=P_t^fH_t\transpose.
\]

\subsection{Continuous-time remarks}

The finite-horizon discrete-time identities in the main text are the cleanest setting for the posterior-as-variational theorem. Continuous-time analogues arise in several ways. In diffusion filtering with continuous observations, feedback particle filters choose a feedback control so that the conditional law of the particles matches the exact posterior in the mean-field limit \citep{yang2013}. From another viewpoint, Schr\"odinger bridge formulations and transport couplings interpret data assimilation as an entropy-regularized measure-transport problem between forecast and analysis laws \citep{reich2019}. These continuous-time theories reinforce the main message of the paper: posterior updates can be expressed as controlled measure transports, and approximate filters differ primarily by the restricted family of transports they permit.